\documentclass[a4paper]{amsart}
\usepackage{amssymb}
\usepackage{ifthen}
\usepackage{graphicx}

\newtheorem{thm}{Theorem}
\newtheorem{cor}{Corollary}

\theoremstyle{definition}

\newcommand{\A}{{\mathcal A}}
\newcommand{\U}{{\mathcal U}}
\newcommand{\es}{{\mathcal S}}

\newcommand{\D}{{\mathbb D}}
\newcommand{\real}{{\operatorname{Re}\,}}

\begin{document}
\bibliographystyle{amsplain}

\title[On the difference of the moduli of the two initial logarithmic coefficients]{On the difference of the moduli of the two initial logarithmic coefficients}

\author[M. Obradovi\'{c}]{Milutin Obradovi\'{c}}
\address{Department of Mathematics,
Faculty of Civil Engineering, University of Belgrade,
Bulevar Kralja Aleksandra 73, 11000, Belgrade, Serbia.}
\email{obrad@grf.bg.ac.rs}

\author[N. Tuneski]{Nikola Tuneski}
\address{Department of Mathematics and Informatics, Faculty of Mechanical Engineering, Ss. Cyril and
Methodius
University in Skopje, Karpo\v{s} II b.b., 1000 Skopje, Republic of North Macedonia.}
\email{nikola.tuneski@mf.edu.mk}

\subjclass[2020]{30C45, 30C50}
\keywords{univalent functions, logarithmic coefficient, sharp results, class $\U$, $\alpha$-convex functions, Ozaki close-to-convex functions}

\begin{abstract}
In this paper, we give sharp bounds of the difference of the moduli of the second and the first logarithmic coefficient for the functions on the class $\U$, for the $\alpha$-convex functions, and for the class $\mathcal{G}(\alpha)$ introduced by Ozaki.
\end{abstract}

\maketitle

\section{Introduction and definitions}
Let in the beginning state some necessary definitions from the theory of univalent functions, as well as previous result that motivated the research presented in this paper.

\medskip

So, functions $f$ analytic in the open unit disk $\D=\{z:|z|<1\}$ and normalized such that $f(0)=f'(0)-1=0$, i.e.,
\begin{equation}\label{e1}
  f(z)=z+a_2z^2+a_3z^3+\cdots,
\end{equation}
form class $\A$. The functions from $\A$ that are one-on-one and onto are the well known univalent functions, and the corresponding class is denoted by  $\es$. The famous Bieberbach conjecture from 1916 (\cite{bieber}), stating that for all functions from $\es$, $|a_n|\le n$, $n=2,3,\ldots$,  proven by de Branges in 1985 (\cite{branges}), was one of the most significant results of the twentieth century and motivated massive research in the area. Most of that research was devoted to estimating coefficients of univalent functions and various functionals of coefficients. More details can be found in \cite{duren, book}. In a similar way, so-called logarithmic coefficients, $\gamma_n$, $n=1,2,\ldots$, were also studied. Those coefficients are defined  by
\begin{equation}\label{eq2}
\log\frac{f(z)}{z}=2\sum_{n=1}^\infty \gamma_n z^n.
\end{equation}
Although coefficients $a_n$ are widely studied, for the logarithmic coefficients very little is known. The Koebe function, with logarithmic coefficients are $1/n$, motivates the conjecture  $|\gamma_n|\le1/n$, $n=2,3,\ldots$, but it is false even in order of magnitude (see Duren \cite[Section 8.1]{duren}), and true only for the class of starlike functions (\cite{der}). For the general class $\mathcal{S}$ the sharp estimates of single logarithmic coefficients are known only for $\gamma_1$ and $\gamma_2$, namely,
$$|\gamma_1|\le1\quad\mbox{and}\quad |\gamma_2|\le \frac12+\frac1e=0.635\ldots.$$
Before continuing, let note that from the relations \eqref{e1} and \eqref{eq2}, after  comparing of coefficients, we receive
\begin{equation}\label{eq3}
 \gamma_{1}=\frac{a_{2}}{2}\quad \text{and}\quad \gamma_{2}=\frac{1}{2}\left(a_3-\frac{1}{2}a_{2}^{2}\right).
\end{equation}

\medskip

Another approach of study of logarithmic coefficients is by considering the difference between the moduli of two consecutive coefficients, over the general class $\es$, or over its subclasses.
Thus, the sharp estimates of $|\gamma _{2}|-|\gamma_{1}|$ for the general class $\es$ were given in \cite{lecko}, with a much simpler proof in \cite{OT_2023-3} obtained using different technique.

\bigskip

\noindent
\textbf{Theorem A.}   \textit{For every $f\in\es$, $-\frac{\sqrt2}{2}\leq |\gamma_{2}|-|\gamma_{1}|\leq\frac12$ holds sharply.}

\bigskip

The differences $|\gamma _{3}|-|\gamma_{2}|$ and $|\gamma _{4}|-|\gamma_{3}|$, again for the general class $\es$ were studied in \cite{OT_2023-5}.

\medskip

In this paper we give sharp estimates of the difference $|\gamma_2|-|\gamma_1|$,
for the class $\U(\lambda)$, for the class of $\alpha$-convex functions ($\mathcal{M}(\alpha)$), and for the class $\mathcal{G}(\alpha)$, defined respectively by:
\[
\begin{split}
\U(\lambda) &= \left\{ f\in\A: \left| \left[  \frac{z}{f(z)}\right]^2 f'(z) -1\right|< \lambda,\, z\in\D\right\}, \,0<\lambda\le1; \\[2mm]
\mathcal{M}(\alpha) &= \left\{ f\in\A: \real\left[(1-\alpha)\frac{zf'(z)}{f(z)} + \alpha\left(1+ \frac{zf''(z)}{f'(z)}\right)\right]  >0,\, z\in \D\right\}, \, \alpha\in \mathbb{R};\\[2mm]
\mathcal{G}(\alpha) &= \left\{ f\in\A: \real\left[1+  \frac{zf''(z)}{f'(z)}\right]  < 1+\frac{\alpha}{2},\, z\in\D\right\}, \, 0<\alpha\le1.
\end{split}
\]
All three classes are in the class $\es$ and play significant role in the theory of univalent functions. Class $\U(\lambda)$ does not belong it the class of normalized starlike univalent functions, nor vise versa. Its earliest origins can be traced back to 1959 in the work of Aksent\'{e}v (\cite{aksentev}) and in past decades received considerable attention, mainly through the work of Obradovi\'{c} and Ponnusamy. The class $\mathcal{M}(\alpha)$ was introduced by Mocanu in 1969 (\cite{mocanu-1969}) and makes a bridge between the classes of starlike and convex functions, while class $\mathcal{G}(\alpha)$ was introduced by Ozaki in 1941 (\cite{ozaki}) and consist of functions convex in one direction. Later, in \cite{jov} it was proven that $\mathcal{G}$ is subclass of the class of starlike functions. More about them can be found in \cite{duren,book}.

\medskip

\section{Estimates for the class $\U(\lambda)$}

In this section we will prove sharp bounds for the difference $|\gamma_2|-|\gamma_1|$, using similar methods as in \cite{OT_2023-3}.

\begin{thm}
Let $f\in\U(\lambda)$ for some $0<\lambda\le1$. Then
\begin{itemize}
  \item[$(i)$] $ -\frac{2\lambda+1}{4}\leq |\gamma_{2}|-|\gamma_{1}|\leq\frac{\lambda}{2}$ if $0\le\lambda\le\frac12$;
  \item[$(ii)$] $ -\frac{\sqrt{2\lambda}}{2}\leq |\gamma_{2}|-|\gamma_{1}|\leq\frac{\lambda}{2}$ if $\frac12\le\lambda\le1$.
\end{itemize}
All estimates are sharp.
\end{thm}

Before proceeding with the proof let note that for $\lambda=1$ we receive the same result as in Theorem A for the class $\es$. That is expected since the extremal functions,
\[f_1(z)=\frac{z}{1-\sqrt2 e^{i\theta}z + e^{2i\theta}z^2}\quad \text{and} \quad f_2(z)=\frac{z}{1 + e^{i\theta}z^2},\]
belong to the class $\U\equiv\U(1)$.
On the other hand, that does not hold for other values of $\lambda\in(0,1)$.

\begin{proof}
In the beginning, let recall that for a function $f\in\U(\lambda)$, $0<\lambda\le1$, in \cite{opon} the following was proven to hold sharply:
\[
|a_3-a_2^2|\le \lambda
\]
and
\begin{equation}\label{e-8}
|a_2|\le 1+\lambda.
\end{equation}

Using this facts and \eqref{eq2} we will obtain the right-hand side inequalities in (i) and (ii). Indeed, since $\frac12|a_2|\le \frac12(1+\lambda)\le1$,
\[
\begin{split}
|\gamma_{2}|-|\gamma_{1}| &\leq |\gamma_2|-\frac12|a_2||\gamma_1|\\
&= \frac12\left| a_3-\frac12a_2^2 \right| - \frac14|a_2|^2 \\
&\le \frac12\left[ \left| \left(a_3-\frac12a_2^2\right) - \frac12a_2^2 \right|\right] \\
&= \frac12|a_3-a_2^2|\\
&\le\frac{\lambda}{2}.
\end{split}
\]
The estimate is sharp as the function $f_3(z) = \frac{z}{1-\lambda e^{i\theta} z^2} = z+\lambda e^{i\theta} z^3+\lambda^2e^{2i\theta}z^5+\cdots$ shows.

\medskip

Next, we will prove the left-hand side inequality in (i). Let $\frac12\le\lambda\le1$. Then, the inequality to be prven is equivalent to
\begin{equation}\label{eq-9}
  \left| a_3-\frac12a_2^2 \right| \ge |a_2|-\sqrt{2\lambda}.
\end{equation}
In the case $|a_2|<\sqrt{2\lambda}$, inequality \eqref{eq-9} is obviously true. For the remaining case, $\sqrt{2\lambda}\le |a_2|\le1+\lambda$, from \eqref{e-8} we receive
\[
\begin{split}
\left| a_3-\frac12a_2^2 \right| &= \left| \left(a_3-a_2^2 \right) + \frac12 a_2^2\right| \ge \frac12|a_2|^2 - |a_3-a_2^2|\\
&\ge \frac12|a_2|^2-\lambda \ge |a_2|-\sqrt{2\lambda}.
\end{split}
\]
The last inequality holds since it is equivalent to
\[ \left(|a_2|-\sqrt{2\lambda}\right)\left(|a_2|+\sqrt{2\lambda}-2\right) \ge 0,\]
having in mind that $|a_2|-\sqrt{2\lambda}\ge0$ and $|a_2|+\sqrt{2\lambda}-2\ge 2(\sqrt{2\lambda}-1)\ge0$ (because $\frac12\le\lambda$ implies $\sqrt{2\lambda}\ge1$). The result os sharp with equality attained for the function
\[f_4(z) = \frac{z}{1-\sqrt{2\lambda}z+\lambda z^2} = z+\sqrt{2\lambda} z^2+\lambda z^3+\cdots.\]

\medskip

We will complete the proof by proving the left-hand inequality in (ii). So, for $0<\lambda\le\frac12$, we have
\[
\begin{split}
|\gamma_{2}|-|\gamma_{1}| &= \frac12\left| a_3-\frac12a_2^2 \right| - \frac12|a_2| \\
&= \frac12 \left| \left(a_3-a_2^2\right) + \frac12a_2^2 \right| - \frac12|a_2|\\
&\ge \frac14|a_2|^2 - \frac12\left| a_3-a_2^2\right| - \frac12|a_2|\\
&\ge \frac14\left( |a_2|^2-2|a_2|-2\lambda \right)\\
&\ge -\frac{2\lambda+1}{4}.
\end{split}
\]
The result is sharp with extremal function
\[ f_5(z) = \frac{z}{1-z+\lambda z^2} = z+z^2+(1-\lambda)z^3+\cdots \]
which, for $0<\lambda\le\frac12$ belongs to $\U(\lambda)$ since in that case $1-z+\lambda z^2\neq0$ for all $z$ in $\D$. The last should be verified separately for $0<\lambda\le\frac14$ and $\frac14<\lambda\le\frac12$.
\end{proof}

\medskip

\section{Estimates for $\alpha$-convex functions}

In this section we will study the case of $\alpha$-convex functions.
In the next two theorems we give some basic properties of the classes (see \cite{SMO_1973-1,SMO_1974-2,MBS_1999-1,book}).

\begin{thm}
For the classes $\mathcal{M}(\alpha)$ the next relations are true:
\begin{itemize}
  \item[$(i)$] $\mathcal{M}(\alpha)\subset \mathcal{S}$ for all real $\alpha$;
  \item[$(ii)$] $\mathcal{M}(\alpha)\subset\mathcal{M}(\beta)\subset\mathcal{M}(0)=\es^*$ for $0<\frac{\alpha}{\beta}<1$;
  \item[$(iii)$]  $\mathcal{M}(\alpha)\subset \mathcal{M}(1)=\mathcal{K}$ for $\alpha>  1$.
\end{itemize}
\end{thm}

Here $\mathcal{S}^*\equiv \mathcal{M}(0)$ and $\mathcal{K}\equiv \mathcal{M}(1)$ are the well-known classes of starlike and convex functions, mapping the unit disk onto a starlike or convex domain, respectively.

\medskip

For the class of $\alpha$-convex functions we also have the following

\begin{thm}
For $\alpha\geq 0$ and $f\in\A$, we have
$$f\in\mathcal{M}(\alpha) \quad \Leftrightarrow \quad F\in \mathcal{S}^{\star},$$
where
\begin{equation}\label{eq5}
F(z)=f(z)\left[\frac{zf'(z)}{f(z)}\right]^{\alpha}.
\end{equation}
\end{thm}

From the equation \eqref{eq5} we easily derive that for $\alpha>0$ and $f\in \mathcal{M}(\alpha)$
there exists $F\in \mathcal{S}^{\star}$, such that
\begin{equation}\label{eq6}
f(z)=\left(\frac{1}{\alpha}\int_{0}^{z}\frac{F^{\frac{1}{\alpha}}(t)}{t}dt\right)^{\alpha},
\end{equation}
and $f(z)=F(z)$ for $\alpha=0.$ For example, for $F(z)=\frac{z}{(1-e^{i\theta}z)^{2}}=k_{\theta}(z)$ (the Koebe function), from \eqref{eq6} we obtain
\[
k_{\theta}(z, \alpha)=
\left(\frac{1}{\alpha}\int_{0}^{z}t^{\frac{1}{\alpha}-1}(1-e^{i\theta}t)^{-\frac{2}{\alpha}}dt\right)^{\alpha}\quad (\alpha>0),
\]
 and
$k_{\theta}(z, \alpha)=k_{\theta}(z)$ for $\alpha=0$. We note that the function $k_{\theta}(z, \alpha)$ is extremal in many problems related with  the class $\mathcal{M}(\alpha),$ $\alpha \geq0$, in the same way as the Koebe function is     for the class $\mathcal{S}$.

\medskip

Further, from the definition of the class $\mathcal{M}(\alpha)$ we have
\[
(1-\alpha)\frac{zf'(z)}{f(z)}+\alpha\left[1+\frac{zf''(z)}{f'(z)}\right]=\frac{1+\omega(z)}{1-\omega(z)},
\]
where $\omega$ is a Schwartz function ($\omega(0)=0$ and $|\omega(z)|<1$ for all $z\in\D$). Using
$f(z)=z+a_{2}z^{2}+a_{3}z^{3}+\cdots,$ and $\omega(z)=c_{1}z+c_{2}z^{2}+\cdots$, after comparing the coefficients, we receive $(1+\alpha)a_{2}=2c_{1}$ and since $|c_{1}|\leq1$,
\[
|a_{2}|\leq \frac{2}{1+\alpha}\quad (\alpha \geq0).
\]
The previous result is the best possible as the functions $f(z)=k_{0}(z, \alpha)$, when $\alpha>0$, and
$f(z)=\frac{z}{(1-z)^{2}}$ for $\alpha=0$, show.

\medskip

Now we give estimates of the difference $|\gamma_{2}|-|\gamma_{1}|$ for the class of $\alpha$-convex functions for $\alpha$ being any non-negative real number. The obtained upper bound is sharp.

\begin{thm}\label{24-th 1}
Let $f\in\mathcal{M}(\alpha)$, $\alpha\geq0$ and let $\gamma_{1}$ and $\gamma_{2}$ be its initial logarithmic coefficients.
\begin{itemize}
  \item[$(i)$] If $0\leq\alpha\leq\frac{1+\sqrt{3}}{2}$, then
  $$ -\frac{1}{\sqrt{2(\alpha^{2}+3\alpha+1)}}\leq |\gamma_{2}|-|\gamma_{1}|\leq\frac{1}{2(1+2\alpha)}.$$
  \item[$(ii)$] If $ \alpha\geq \frac{1+\sqrt{3}}{2},$  then
  $$ -\frac{6\alpha ^2+10\alpha +3}{4 (2 \alpha +1) (\alpha^2+3\alpha +1)} \leq |\gamma_{2}|-|\gamma_{1}|
\leq\frac{1}{2(1+2\alpha)}.$$
\end{itemize}
Given upper bounds are sharp.
\end{thm}

\begin{proof}
Let put $$J(f,\alpha;z)=(1-\alpha)\frac{zf'(z)}{f(z)}+\alpha\left[1+\frac{zf''(z)}{f'(z)}\right].$$
Then, by definition of the class $\mathcal{M}(\alpha)$,
\[
\begin{split}
H(f,\alpha;z) &=\frac{1-J(f,\alpha;z)}{1+J(f,\alpha;z)}\\
&=-\frac{1}{2}(1+\alpha)a_{2}z-
\left[(1+2\alpha)a_{3}-\frac{\alpha^{2}+8\alpha+3}{4}a_{2}^{2}\right]z^{2}+\cdots
\end{split}
\]
is a Schwartz function, and by using the inequalities $|c_{1}|\leq1$ and $|c_{2}|\leq 1-|c_{1}|^{2},$
 for $\alpha\geq0$ we receive that $\frac{1}{2}(1+\alpha)|a_{2}|\leq1$ and
$$\left|(1+2\alpha)a_{3}-\frac{\alpha^{2}+8\alpha+3}{4}a_{2}^{2}\right|\leq1-\frac{1}{4}(1+\alpha)^{2}|a_{2}|^{2}.$$
From these inequalities for $\alpha\geq0$ we obtain  $|a_{2}|\leq\frac{2}{1+\alpha}$ and
\begin{equation}\label{eq10}
\left|a_{3}-\frac{\alpha^{2}+8\alpha+3}{4(1+2\alpha)}a_{2}^{2}\right|
\leq\frac{1}{(1+2\alpha)}-\frac{(1+\alpha)^{2}}{4(1+2\alpha)}|a_{2}|^{2}.
\end{equation}
Here, we can note that from \eqref{eq10}, for $\alpha=0$ (the starlike case) we have
$$\left|a_{3}-\frac{3}{4}a_{2}^{2}\right|\leq 1-\frac{1}{4}|a_{2}|^{2},$$
and for  $\alpha=1$ (the convex case):
$$\left|a_{3}-a_{2}^{2}\right|\leq\frac{1}{3} (1-|a_{2}|^{2})$$
(Trimble \cite{trimble}).

\medskip

For the upper bounds (right estimates) in both cases (i) and (ii), using \eqref{eq3} and \eqref{eq10}, we have for
$\alpha\geq0$:
\[
\begin{split}
|\gamma _{2}|-|\gamma_{1}|&=\frac{1}{2}\left|a_3-\frac{1}{2}a_{2}^{2}\right|-\frac{1}{2}|a_{2}|\\
&=\frac{1}{2}\left|\left(a_{3}-\frac{\alpha^{2}+8\alpha+3}{4(1+2\alpha)}a_{2}^{2}\right)+
\frac{\alpha^{2}+4\alpha+1}{4(1+2\alpha)}a_{2}^{2}\right|-\frac{1}{2}|a_{2}|\\
&\leq\frac12\left|a_{3}-\frac{\alpha^{2}+8\alpha+3}{4(1+2\alpha)}a_{2}^{2}\right|+\frac{\alpha^{2}+4\alpha+1}{8(1+2\alpha)}|a_{2}|^{2}
-\frac{1}{2}|a_{2}|\\
&\leq \frac{1}{2(1+2\alpha)}-\frac{(1+\alpha)^{2}}{8(1+2\alpha)}|a_{2}|^{2}+\frac{\alpha^{2}+4\alpha+1}{8(1+2\alpha)}|a_{2}|^{2}
-\frac{1}{2}|a_{2}|\\
&=\frac{1}{2(1+2\alpha)}+\frac{\alpha}{4(1+2\alpha)}|a_{2}|^{2}-\frac{1}{2}|a_{2}|\\
&\leq\frac{1}{2(1+2\alpha)},
\end{split}
\]
since
$$\frac{\alpha}{4(1+2\alpha)}|a_{2}|^{2}-\frac{1}{2}|a_{2}|\leq0 $$
for $0\leq|a_{2}|\leq\frac{2}{1+\alpha}.$
The result is sharp as the functions (see \eqref{eq6}) defined by
$$f(z)=\left(\frac{1}{\alpha}\int_{0}^{z}t^{\frac{1}{\alpha}-1}(1-t^{2})^{-\frac{1}{\alpha}}dt\right)^{\alpha}
=z+\frac{1}{1+2\alpha}z^{3}+\cdots  \quad (\text{for } \alpha>0),$$
and $f(z)=\frac{z}{1-z^{2}}$ (for $\alpha=0$) show.\\

($i$) Now, for the lower bound in the first case, let $0\leq\alpha\leq\frac{1+\sqrt{3}}{2}$. Then the left hand side of inequality to be proven is equivalent to
$$\frac{1}{2}\left|a_3-\frac{1}{2}a_{2}^{2}\right|-\frac{1}{2}|a_{2}|\geq-\frac{1}{\sqrt{2(\alpha^{2}+3\alpha+1)}},$$
i.e., to
$$\left|a_3-\frac{1}{2}a_{2}^{2}\right|\geq|a_{2}|-\sqrt{\frac{2}{\alpha^{2}+3\alpha+1}}.$$

\medskip

If $0\leq|a_{2}|<\sqrt{\frac{2}{\alpha^{2}+3\alpha+1}}$, then the previous inequality is true.

\medskip

Further, let $\sqrt{\frac{2}{\alpha^{2}+3\alpha+1}}\leq |a_{2}|\leq \frac{2}{1+\alpha}.$ Then using
\eqref{eq3} and \eqref{eq10} we have
\begin{equation}\label{eq 12}
\begin{split}
|\gamma _{2}|-|\gamma_{1}|&=\frac{1}{2}\left|a_3-\frac{1}{2}a_{2}^{2}\right|-\frac{1}{2}|a_{2}|\\
&=\frac{1}{2}\left|\left(a_{3}-\frac{\alpha^{2}+8\alpha+3}{4(1+2\alpha)}a_{2}^{2}\right)+
\frac{\alpha^{2}+4\alpha+1}{4(1+2\alpha)}a_{2}^{2}\right|-\frac{1}{2}|a_{2}|\\
&\geq\frac{\alpha^{2}+4\alpha+1}{8(1+2\alpha)}|a_{2}|^{2}-
\frac12\left|a_{3}-\frac{\alpha^{2}+8\alpha+3}{4(1+2\alpha)}a_{2}^{2}\right|
-\frac{1}{2}|a_{2}|\\
&\geq\frac{\alpha^{2}+4\alpha+1}{8(1+2\alpha)}|a_{2}|^{2}-\frac{1}{2(1+2\alpha)}+\frac{(1+\alpha)^{2}}{8(1+2\alpha)}|a_{2}|^{2}
-\frac{1}{2}|a_{2}|\\
&=\frac{\alpha^{2}+3\alpha+1}{4(1+2\alpha)}|a_{2}|^{2}-\frac{1}{2}|a_{2}|-\frac{1}{2(1+2\alpha)}.
\end{split}
\end{equation}
From the previous relation, it is enough to show that
$$\frac{\alpha^{2}+3\alpha+1}{4(1+2\alpha)}|a_{2}|^{2}-\frac{1}{2}|a_{2}|-\frac{1}{2(1+2\alpha)}
\geq -\frac{1}{\sqrt{2(\alpha^{2}+3\alpha+1)}},$$
which is equivalent to
$$\left[|a_{2}|-\sqrt{\frac{2}{\alpha^{2}+3\alpha+1}}\right]
\left[\frac{\alpha^{2}+3\alpha+1}{2(1+2\alpha)}\left(|a_{2}|+\sqrt{\frac{2}{\alpha^{2}+3\alpha+1}}\right)-1\right]\geq0.$$
The last inequality is indeed true since by assumption $\sqrt{\frac{2}{\alpha^{2}+3\alpha+1}}\leq |a_{2}|\leq \frac{2}{1+\alpha}$ and
$$\frac{\alpha^{2}+3\alpha+1}{2(1+2\alpha)}\left(|a_{2}|+\sqrt{\frac{2}{\alpha^{2}+3\alpha+1}}\right)-1
\geq \frac{\sqrt{2(\alpha^{2}+3\alpha+1)}}{1+2\alpha}-1\geq0,$$
since also $0\leq\alpha\leq\frac{1+\sqrt{3}}{2}$.

\medskip

Finally, let $\alpha\geq\frac{1+\sqrt{3}}{2}$. Then, using the beginning and end of the relation
\eqref{eq 12} we get
\[
\begin{split}
|\gamma _{2}|-|\gamma_{1}|&
\geq\frac{\alpha^{2}+3\alpha+1}{4(1+2\alpha)}|a_{2}|^{2}-\frac{1}{2}|a_{2}|-\frac{1}{2(1+2\alpha)}\\
&\geq -\frac{6\alpha ^2+10\alpha +3}{4 (2 \alpha +1) (\alpha^2+3\alpha +1)},
\end{split}
\]
since the last function attains its minimum for $|a_{2}|_{0}=\frac{1+2\alpha}{\alpha^{2}+3\alpha+1}$ and since
$$\sqrt{\frac{2}{\alpha^{2}+3\alpha+1)}}\leq |a_{2}|_{0}<\frac{2}{1+\alpha}$$
for $\alpha\geq\frac{1+\sqrt{3}}{2}$.
\end{proof}

\medskip

\begin{cor}
For $\alpha=0$, i.e., for the starlike functions $\mathcal{S}^{\star}$, as for the class
$\mathcal{S}$, we have
$$-\frac{1}{\sqrt{2}}\leq |\gamma_{2}|-|\gamma_{1}|\leq\frac{1}{2}.$$
(see \cite{lecko,OT_2023-3}). Both inequalities are sharp.
For $\alpha=1$, i.e., for the convex function from the class  $\mathcal{K}$, we have
$$-\frac{1}{\sqrt{10}}\leq |\gamma_{2}|-|\gamma_{1}|\leq\frac{1}{6}.$$
(see \cite{OT_2023-3}). The right hand side of the inequality is sharp.
\end{cor}

\medskip

\section{Estimates for the class $\mathcal{G}(\alpha)$}

In the last section we will give estimate of the difference (sharp for the right-hand side) of $|\gamma_{2}|-|\gamma_{1}|$ for the class $\mathcal{G}(\alpha)$.

\begin{thm}
  If $f\in \mathcal{G}(\alpha)$, $0<\alpha\le1$, then
\[ -\frac{\alpha(17-\alpha)}{12(8-\alpha)} \le |\gamma_{2}|-|\gamma_{1}| \le \frac{\alpha}{12}. \]
The right-hand side inequality is sharp.
\end{thm}

\begin{proof}
Let $f\in \mathcal{G}(\alpha)$, $0<\alpha\le1$. Then
\[ \real\left[1 - \frac{2}{\alpha}\frac{zf''(z)}{f'(z)}\right] >0, \quad z\in\D. \]
So, for $h(z) = 1 - \frac{2}{\alpha}\frac{zf''(z)}{f'(z)}$, we have $\real h(z)>0$, $z\in\D$, and for $\Phi(z)=\frac{1-h(z)}{1+h(z)}$, $|\Phi(z)|<1$, $z\in\D$, and
\[ \Phi(z) = \frac{zf''(z)}{\alpha f'(z)-zf''(z)} = \frac{2}{\alpha}a_2z + \frac{6}{\alpha}\left( a_3+\frac23\frac{1-\alpha}{\alpha}a_2^2 \right)z^2+\cdots. \]
The last relation implies (since $\Phi(z)|<1$, $z\in\D$) that
\[  \frac{6}{\alpha}\left| a_3+\frac23\frac{1-\alpha}{\alpha}a_2^2 \right| \le 1-\left| \frac{2}{\alpha}a_2 \right|^2, \]
i.e.,
\begin{equation}\label{e-11}
  \left| a_3+\frac23\frac{1-\alpha}{\alpha}a_2^2 \right| \le \frac{1}{6\alpha}   \left(\alpha^2-4|a_2|^2 \right).
\end{equation}

\medskip

Let briefly note that this inequality, for $\alpha=1$ leads to the estimate
\[ |a_3|\le \frac16(1-4|a_2|^2), \]
holding on the class $\mathcal{G}$ and improving (in the case $n=3$) the one given in \cite{opon-2} (Theorem 1) stating that
\begin{equation}\label{A}
|a_n|\le \frac{\alpha}{n(n-1)}\quad n=2,3,\ldots.
\end{equation}

\medskip

Continuing the proof, we use relations \eqref{eq3} and \eqref{e-11} we receive
\[
\begin{split}
|\gamma _{2}|-|\gamma_{1}|&=\frac{1}{2}\left|a_3-\frac{1}{2}a_{2}^{2}\right|-\frac{1}{2}|a_{2}|\\
&=\frac{1}{2}\left|\left(a_{3} + \frac23\frac{1-\alpha}{\alpha}a_{2}^{2}\right) -
\frac{4-\alpha}{6\alpha}a_{2}^{2}\right|-\frac{1}{2}|a_{2}|\\
&\le \frac12 \left| a_3+\frac23\frac{1-\alpha}{\alpha}a_2^2 \right| +\frac{4-\alpha}{12\alpha}|a_2|^2 - \frac12|a_2|\\
&\le\frac{1}{12\alpha}\left(\alpha^2-4|a_2|^2\right) +\frac{4-\alpha}{12\alpha}|a_2|^2 - \frac12|a_2|\\
&= \frac{1}{12\alpha}\left(-\alpha|a_2|^2 -6\alpha|a_2|+\alpha^2\right) \\
&\le \frac{1}{12\alpha}\alpha^2 = \frac{1}{12\alpha}.
\end{split}
\]
This result is sharp for the function $f$ defined by $f'(z)=(1-z^2)^{\alpha/2}=1-\frac{\alpha}{2}z^2+\cdots$, such that $f(z)=z-\frac{\alpha}{6}z^3+\cdots$.

\medskip

For the left-hand side inequality, using \eqref{e-11} and the previous considerations, we receive
\[
\begin{split}
|\gamma _{2}|-|\gamma_{1}|&\ge \frac{1}{2}\left( \frac{4-\alpha}{6\alpha}|a_2|^2-\left|a_3+\frac23\frac{1-\alpha}{\alpha}a_2^2\right|\right)-\frac12|a_2|\\
&\ge \frac{1}{2}\left[ \frac{4-\alpha}{6\alpha}|a_2|^2 - \frac{1}{6\alpha}\left( \alpha^2-4|a_2|^2 \right)\right]-\frac12|a_2|\\
&= \frac{1}{12\alpha}\left[ (8-\alpha)|a_2|^2 - 6\alpha|a_2|-\alpha^2\right] = \varphi(|a_2|),
\end{split}
\]
where $\varphi(t)= \frac{1}{12\alpha}\left[ (8-\alpha)t^2 - 6\alpha t-\alpha^2\right] $. Inequality \eqref{A} for $n=2$ gives $|a_2|\in[0,\alpha/2]$, so after verifying that for $0<\alpha\le1$, $\varphi$ attains its minimal value on $[0,\alpha/2]$ for $t_0=\frac{3\alpha}{8-\alpha}<\frac{\alpha}{2}$, we conclude
\[ |\gamma _{2}|-|\gamma_{1}|\ge \varphi(t_0) = \frac{1}{12\alpha}\left( -\frac{9\alpha^2}{8-\alpha} - \alpha^2\right) =-\frac{\alpha(17-\alpha)}{12(8-\alpha)}  .\]

\medskip

This compeltes the proof.
\end{proof}

For $\alpha=1$ in the previous theorem we receive the following

\begin{cor}
  If $f\in \mathcal{G}\equiv \mathcal{G}(1)$, $0<\alpha\le1$, then
\[ -\frac{4}{21} \le |\gamma_{2}|-|\gamma_{1}| \le \frac{1}{12}. \]
The right-hand side inequality is sharp. The function $f(z)=z-\frac12z^2$ from $\mathcal{G}$, such that $|\gamma_{2}|-|\gamma_{1}| = -\frac{3}{16}=-0.1875\ldots>-\frac{4}{21}=-0.19047\ldots$, implies that the sharp lower bound is close to the obtained one.
\end{cor}

\end{document}